\documentclass[11pt]{amsart}
\usepackage{amsmath,amsfonts,amssymb,amsthm,amscd,latexsym,euscript,verbatim,
bbold}
\usepackage[all]{xy}

\makeatletter
\def\section{\@startsection{section}{1}%
  \z@{1.1\linespacing\@plus\linespacing}{.8\linespacing}%
  {\normalfont\Large\scshape\centering}}
\makeatother

\theoremstyle{plain}

\newtheorem*{thmA}{Theorem A}
\newtheorem*{thmB}{Theorem B}

\newtheorem*{conj*}{Root Groups Conjecture}
\newtheorem*{thm1.2}{(1.2) Theorem}
\newtheorem*{thm1.3}{(1.3) Theorem}
\newtheorem*{thm1.4}{(1.4) Theorem}
\newtheorem*{prop*}{Proposition}
\newtheorem*{thm*}{Theorem}

\def\eroman{\etype{\roman}}

\newtheorem{prop}{Proposition}[section]

\newtheorem{thm}[prop]{Theorem}
\newtheorem{cor}[prop]{Corollary}
\newtheorem{lemma}[prop]{Lemma}

\theoremstyle{definition}
\newtheorem{Def}[prop]{Definition}
\newtheorem*{Def*}{Definition}
\newtheorem{Defs}[prop]{Definitions}

\newtheorem{examples}[prop]{Examples}

\newtheorem{notation}[prop]{Notation}
\newtheorem*{notation*}{Notation}
\newtheorem{remark}[prop]{Remark}

\newcommand{\etype}[1]{\renewcommand{\labelenumi}{(#1{enumi})}}

\newcommand{\la}{\lambda}

\newcommand{\ff}{F}

\newcommand{\zz}{\mathbb{Z}}

\newcommand{\ga}{\alpha}
\newcommand{\gb}{\beta}
\newcommand{\gc}{\gamma}

\newcommand{\gd}{\delta}

\newcommand{\gre}{\epsilon}
\newcommand{\gl}{\lambda}

\newcommand{\gvp}{\varphi}
\newcommand{\gr}{\rho}
\newcommand{\gs}{\sigma}

\newcommand{\gt}{\tau}

\newcommand{\charc}{{\rm char}}

\newcommand{\half}{\textstyle{\frac{1}{2}}}

\numberwithin{equation}{section}

\begin{document}
\title[Axes of Jordan type]{Axes of Jordan type in non-commutative algebras}
\author[Louis Rowen, Yoav Segev]
{Louis Rowen\qquad Yoav Segev}

\address{Louis Rowen\\
         Department of Mathematics\\
         Bar-Ilan University\\
         Ramat Gan\\
         Israel}
\email{rowen@math.biu.ac.il}
\address{Yoav Segev \\
         Department of Mathematics \\
         Ben-Gurion University \\
         Beer-Sheva 84105 \\
         Israel}
\email{yoavs@math.bgu.ac.il}
\thanks{$^*$The first author was supported by the Israel Science Foundation grant 1623/16}


 \begin{abstract}
The Peirce decomposition of a Jordan algebra with respect to an
idempotent is well known.  This decomposition was taken one step
further and generalized recently by Hall, Rehren and Shpectorov, with
their introduction of {\it axial algebras}, and in particular {\it primitive axial
algebras of Jordan type} (PJs for short). It turns out that these notions are closely
related to $3$-transposition groups and vertex operator algebras. De
Medts, Peacock, Shpectorov, and M. Van Couwenberghe generalized axial algebras
to {\it decomposition algebras} which, in particular, are not
necessarily commutative. This paper deals with
decomposition algebras which are
non-commutative versions of PJs.
\end{abstract}

\date{\today}
\maketitle

\section{Introduction}
Axial algebras were introduced recently by Hall, Rehren and
Shpectorov in \cite{HRS}. These algebras are of interest because of
their connection with group theory and with vertex operator
algebras. We refer the reader to the introduction of \cite{HRS} for
further information. This notion was put in the more general
(non-commutative) framework of decomposition algebras in
\cite[Definitions~4.1 and ~5.1]{DPSC}.

This paper studies a certain class of decomposition algebras,
extending the notion of ``primitive axial algebras of Jordan type''
(see \cite{HRS}) to the non-commutative setting. Thus in this paper
$A$ is an algebra not necessarily associative or commutative, not
necessarily with a multiplicative unit element, over a field $F$
with $\charc(F)\ne 2,$ unless stated otherwise. As in \cite{HRS} we
are interested in such algebras generated by certain kind of axes.
 These are idempotents,
as defined in Definition \ref{defs main3}. As in \cite{Sa} and \cite{HRS},
most of our efforts are devoted to classifying our $2$-generated algebras.
See Theorem A and B below.

An additional underlying motivation in studying axes is to understand the
``reason" that a primitive axial algebra of Jordan type
$\eta\ne1/2,$ which is generated (as an algebra) by a finite set of
$m$ axes, is finite dimensional. This fact follows from the
classification of finite simple groups, since such algebras, in the
commutative setting, come from $3$-transposition groups.  For
$\eta=1/2$ and $m>3,$ it is not known whether this is true or not.
But this goal may lie in the not-so-near future.
This will be further pursued in \cite{RS}.

\begin{Defs}\label{defs main20}$ $
\begin{enumerate}
\item  The algebra $A$ is {\it flexible} if it
satisfies the identity $(xy)x=x(yx).$

Note that
commutative algebras are flexible. Indeed
\[
(xy)x=(yx)x=x(yx).
\]

\item
We often denote multiplication $x\cdot y$ in~$A$ by
juxtaposition: $x y$.

\item
The commutative center of $A$ is denoted $Z(A)$ and defined
as
\[
Z(A)=\{x\in A\mid xy=yx, \text{ for all }y\in A\}.
\]

\item
We define {\it left and right} multiplication maps
$L_a(b) := a\cdot b $ and $R_a(b) := b\cdot a.$

\item
We write $A_{\la}(X_a)$ for the eigenspace
of $\la\in\ff$ with respect to the transformation $X_a$,
$X\in \{L,R\},$ i.e., $A_{\la}(L_a) = \{ v \in A: a\cdot v = \la v\},$
and similarly for $A_{\la}(R_a).$
{\bf Often we just write $A_\la$ for
$A_{\gl}(L_a)$, when $a$ is understood.}

\item
We denote: $A_{\gl,\gd}(a):= A_{\gl}(L_a)\cap A_{\gd}(R_a).$ An
element in $A_{\gl,\gd }(a)$ will be called a $(\gl,\gd)$-{\it
eigenvector} of $a,$  and $(\gl,\delta)$ will be called its {\it
eigenvalue}.
{\bf We just write $A_{\gl,\gd }$ when the idempotent $a$ is  understood.}
\end{enumerate}
\end{Defs}

\begin{Def}\label{defs main3}$ $
\smallskip
Let $a\in A$ be an idempotent, and $\gl,\gd\notin\{0,1\}$ in $\ff$.
\begin{enumerate}
\item
 $a$ is a {\it left axis of type $\gl$} if
\begin{itemize}
\item[(a)]
$a$ is {\it absolutely left primitive}, that is, $A_1(L_a) = \ff a$.

\item[(b)]
$(L_a-\gl)(L_a-1)L_a =0.$

\item[(c)]
The direct sum decomposition of
$A$ into eigenspaces of $L_a$ is a $\zz_2$-grading:
\[
A=\overbrace{A_{0}\oplus A_1}^{\text {$+$-part}}\oplus
\overbrace{A_{\gl}}^{\text{$-$-part}},
\]
recalling that $A_{\gl}$ means $A_{\gl}(L_a)$.
\end{itemize}

\item
A {\it right axis of type $\gl$} is defined similarly.

\item
$a$ is an {\it axis (2-sided) of type $(\gl,\gd)$} if $a$ is a left
axis of  type $\gl$ and a    right axis of type $\gd$ and, in
addition,
 $L_a R_a =  R_a  L_a.$
Thus $A_{1,1} =\ff a =A_1$; in particular $A_{1,\gd} = A_{\gl,1} =
0$. Hence $A$ decomposes into a
direct sum
\[
A=\overbrace{A_{1,1} \oplus A_{0,0}}^{\text {$++$-part}}\oplus
\overbrace{A_{0,\gd}}^{\text{$+-$-part}} \oplus \overbrace{
A_{\la,0}}^{\text {$-+$-part}}\oplus \overbrace{A_{\gl,\gd
}}^{\text{$--$-part}},
\]
and this is $\zz_2\times\zz_2$ grading of $A$ (multiplication $(\gre,\gre')(\gr,\gr')$
is defined in the obvious way for $\gre,\gre',\gr,\gr'\in\{+,-\}$),

\item
The axis $a$  is of {\it Jordan type $(\gl,\gd)$} if
$A_{\gl,0} = A_{0,\gl}=0.$

In this case,
\[
A=\overbrace{A_{1,1} \oplus A_{0,0}}^{\text {$+$-part}}\oplus
\overbrace{A_{\gl,\gd }}^{\text{$-$-part}},
\]
is a $\zz_2$-grading of $A.$
%
%
\end{enumerate}
\end{Def}

\noindent
See Examples \ref{flex} for non-commutative
algebras generated by two axes of Jordan type.

Note that besides being non-commutative,   Definition \ref{defs
main3}(1c) {\bf does not} require the condition that $A_{0}$ be a
subalgebra, contrary to the usual hypothesis in the
theory of primitive axial algebras of Jordan type. But this condition does not seem
to pertain to any of the proofs.

We can now state our main results.

\begin{thmA}[Theorem \ref{thm dim 2}]\label{thm dim2}
Suppose that $\dim(A)=2$ and $A$ is generated
by two axes $a$ and $b.$  Then either $A$ is one of the
commutative algebras described in \cite[Lemma 3.1.2, p.~269]{HSS},
or $A$ is the algebra of Example~\ref{flex}(i).
In particular, $A$ is flexible.
\end{thmA}

\begin{thmB}[Theorem \ref{thm 2-gen}]\label{thm dim3}
Suppose $A$ is generated by two axes of Jordan type.
Then either $A$ is a primitive axial algebra of Jordan type
as in \cite{HRS} (in particular $A$ is commutative),
or $A$ is as in Examples~\ref{flex}.  In particular
$\dim(A)\le 3,$ and $A$ is flexible.
\end{thmB}
\medskip

\noindent This extends the characterization of $2$-generated
primitive axial algebras of Jordan type in \cite{HRS}. We note also
that although our fusion rules do not require that
$A_0(a)^2\subseteq A_0(a)$, for an axis $a,$ this holds if $A$ is
generated by two axes of Jordan type (see Corollary \ref{co 020}).

\section{Algebras generated by two axes}\label{nc1}
In this section we give some basic properties
of axes.  In Theorem \ref{thm dim 2} we classify
the algebras of dimension $2$ generated by two axes and
in Theorem \ref{thm 2-gen} we classify the algebras
generated by two axes of Jordan type.

\begin{notation}\label{not00}
Let $a$ be an axis of type $(\gl,\gd).$ Given $y\in A,$ we write  $y= \ga_y  a+y_0+ y_{\gl},$ where $y _\rho \in A_{\rho}(L_a),$
$\rho \in \{0,\gl\}$, and $\ga \in \ff$. In this notation,
$y_0 = y_{0,0}+y_{0,\gd},$ and $y_\gl =
y_{\gl,0}+y_{\gl,\gd}.$ If $a$ is of  Jordan type, then
$y_0 = y_{0,0}$ and $y_\gl = y_{\gl,\gd}.$

Similarly, we write $y=\ga_y a+{}_{0}{y}+{}_{\gd}y,$ where ${}_{\gr}{y}\in A_{\gr}(R_a).$
\end{notation}

\begin{remark}\label{gl=gd}
If $a\in A$ is an axis such that $A=A_{1,1}+A_{0,0},$ then
our convention is that $a$ is an axis of Jordan type $(\gl,\gl),$
for any $\gl\in\ff,$ with $\gl\notin\{0,1\}.$
\end{remark}

Here are some basic observations.

\begin{lemma}\label{lem basic}$ $
Suppose  $a$ is an axis of type $(\gl,\gd),$ and let $y\in A.$
\begin{enumerate}
\item
$a(\ff a+A_0)= \ff a .$

\item
$ay=\ga_y  a+\gl y_{\gl},$ implying
$y_{\gl} =\frac 1 {\gl}(ay -\ga_y  a).$

\item
$a(ay)=\ga_y (1-\gl) a+\gl ay.$

\item
If $a$ is of Jordan type, then $ya\in \ff a+\ff ay.$

\item
An axis $a$ of Jordan type $(\gl,\gd)$ is in $Z(A),$ iff $\gl=\gd.$
\end{enumerate}
\end{lemma}
\begin{proof}
(1) This is obvious.

(2) Apply $L_a$ to  $y$ (and use Notation~\ref{not00}).

(3) $ a(ay)=\ga_y  a+\gl^2 y_{\gl}=\ga_y  a+\gl \textstyle{(ay-\ga_y
a)},$  yielding (3).

(4) By (2), $y_{\gl,\gd}\in \ff a+ \ff ay,$ and
$ya=\ga_ya+\gd y_{\gl,\gd}.$

(5) Suppose that $a\in Z(A).$  If $A_{\gl,\gd}=0,$ then, by Remark \ref{gl=gd}, $\gl=\gd.$
Otherwise let  $0\ne x_{\gl,\gd}\in A_{\gl,\gd }.$
Then $\gl x=ax=xa=\gd x,$ so $\gl=\gd.$
Conversely, for $y\in A.$
Then $ay=\ga a+\gl x_{\gl,\gl}=xa,$ so $a\in Z(A).$
\end{proof}

\begin{lemma}\label{lem ab notin ffa}$ $
\begin{enumerate}\eroman
\item $ab \notin \ff a$ or $ab =0$, for any   idempotent $a$ and a  right axis $b$.

\item
$ba \notin \ff a$ or $ba =0$, for any   idempotent $a$ and a left axis $b$.

\item For any left axis $a$ of type
$\gl$ and any   right axis $b$, either $ab =0$ or
$b_{\gl}\ne 0.$

\item If $a$ is of Jordan type $({\gl,\gd})$ with $\gl\ne \gd,$
and $b$ is an axis, then $b_{\gl,\gd}^2 = 0.$
\end{enumerate}
\end{lemma}
\begin{proof}
(i)\&(ii) If $ab = \gc a,$ then $a$ is an eigenvector of $R_b,$
with eigenvalue $\gc.$  If $\gc=1,$ then $b$ is not absolutely primitive,
a contradiction.  Otherwise, by the fusion rules for $b,$
$a=a^2\in\ff b+A_0(b).$  But then $ab\in\ff b.$ This together with $ab\in\ff a,$
forces $ab=0.$  The proof of part (ii) is the same.

(iii) If $b_{\gl}= 0,$ then $ab = \ga_b a,$ so we are done by (i).

(iv)  Recall that $L_bR_b=R_bL_b.$ We compute $b(ab)=(ba)b.$
\[
\begin{aligned}
b(ab) = b(\ga_b a+ \gl b_{\gl,\gd} )
& =  \ga_b ba + \gl b b_{\gl,\gd}\\
& =\ga_b (\ga_b a+ \gd b_{\gl,\gd}) +\gl ( b_{\gl,\gd}
+ (b_0 + \ga_b a)) b_{\gl,\gd}\\
& =\ga_b ^2 a +\gl b_{\gl,\gd}^2 +\ga_b(\gl^2 +\gd )b_{\gl,\gd}+\gl b_0  b_{\gl,\gd}.
\end{aligned}
\]
On the other hand,
\[
\begin{aligned}(ba)b &= (\ga_b a+ \gd b_{\gl,\gd}) b = \ga_b ab+\gd b_{\gl,\gd} b\\
& =  \ga_b (\ga_b a+ \gl b_{\gl,\gd} ) + \gd b_{\gl,\gd}(\ga_b a+b_0+ b_{\gl,\gd})\\
& =\ga_b^2 a + \gd b_{\gl,\gd}^2 +\ga_b( \gl +\gd ^2)b_{\gl,\gd}+ \gd
b_{\gl,\gd}b_0,
\end{aligned}
\]
so matching parts in $A_0+\ff a$
shows $\ga_b^2 a + \gl b_{\gl,\gd}^2 = \ga_b^2 a + \gd
b_{\gl,\gd}^2,$ and thus $b_{\gl,\gd}=0.$
\end{proof}

\begin{thm}\label{thm dim 2}
Suppose that $\dim(A)=2$ and $A$ is generated
by two axes $a$ and $b.$  Then either $A$ is one of the
algebras described in \cite[Lemma 3.1.2, p.~269]{HSS}
(in particular $A$ commutative)
or $A$ is the non-commutative flexible algebra of Example \ref{flex}(i).
\end{thm}
\begin{proof}
Let $a$ be of type
$(\gl,\gd)$ and $b$ be of type $(\gl',\gd').$
Assume first that $A_{\gl}(L_a)=0.$  Then $ab=\ga_b a,$
hence $ab=0,$ by Lemma \ref{lem ab notin ffa}(i).

Suppose $A_{\gd}(R_a)\ne 0.$  Write $b=\ga_ba+b_{\gd},$ with $b_{\gd}\in A_{\gd}(R_a).$
Then, an easy computation shows that $ba=\ga_b(1-\gd)a+\gd b.$  Then, by the fusion
rules
\[
\ff a \ni b_{\gd}^2=b-\ga_b ba+\ga_b^2a=b-\gd b+\gc a, \gc\in\ff.
\]
Since $\gd\ne 1,$ this
is impossible.
Hence, as above $ba=0,$ and $A$
is the algebra $2B$ of \cite{HRS}.  Thus we may assume that $A_{\gl}(L_a)\ne 0,$
and similarly $A_{\gd}(R_a)\ne 0.$  Thus, as $A$ is $2$-dimensional, $A_0(L_a)=A_0(R_a)=0,$
and $A_{\gl,\gd}(a)\ne 0.$  By symmetry the same holds for $b.$

Write
\[
b=\ga_b a+b_{\gl,\gd},\quad a=\ga_a b+a_{\gl',\gd'}.
\]
Then (also by symmetry)
\[
\begin{aligned}
&b_{\gl,\gd}=-\ga_b a+b,\qquad a_{\gl',\gd'}=a-\ga_a b,\\
&ab=\ga_b(1-\gl)a+\gl b=\gd' a+\ga_a(1-\gd')b.\\
&ba=\ga_a(1-\gl')b+\gl' a=\gd b+\ga_b(1-\gd)a.
\end{aligned}
\]
We conclude, also using symmetry, that
\[
\begin{aligned}
&ab=\gd' a+\gl b\\
&ba=\gl' a+\gd b\\
&\textstyle{\ga_b=\frac{\gd'}{1-\gl}=\frac{\gl'}{1-\gd}.}\\
&\textstyle{\ga_a=\frac{\gd}{1-\gl'}=\frac{\gl}{1-\gd'}.}
\end{aligned}
\]
Hence
\[
\gl\gl'-\gd\gd'=\gl'-\gd'=\gl-\gd.
\]
Now
\[
\ff a\ni b_{\gl,\gd}^2=(-\ga_b a+b)^2=\ga_b^2a-\ga_b(ab+ba)+b=\ga_b^2a-\ga_b(\gl'+\gd')a-\ga_b(\gl+\gd)b+b.
\]
Hence,
\[
\begin{aligned}
&\ga_b=\frac{1}{\gl+\gd}\implies \frac{\gl'}{1-\gd}=\frac{1}{\gl+\gd}\implies\\
&\gl\gl'+\gl'\gd=1-\gd,\quad \text{and by symmetry}\\
&\gl'\gl+\gl\gd'=1-\gd'.
\end{aligned}
\]
Also we have
\[
\ff b\ni (a-\ga_ab)^2=a-\ga_a(\gl'+\gd')a-\ga_a(\gl+\gd)b+\ga_a^2 b.
\]
Hence
\[
\ga_a=\frac{1}{\gl'+\gd'}=\frac{\gd}{1-\gl'}.
\]
So
\[
\begin{aligned}
&\gl'\gd+\gd\gd'=1-\gl',\qquad\text{and by symmetry}\\
&\gl\gd'+\gd\gd'=1-\gl.
\end{aligned}
\]
 We thus have the following equalities
\begin{align}\label{1}
&\gl\gl'-\gd\gd'=\gl'-\gd'.\\\label{2}
&\gl\gl'-\gd\gd'=\gl-\gd.\\\label{3}
&\gl\gl'+\gl'\gd=1-\gd.\\\label{4}
&\gl'\gl+\gl\gd'=1-\gd'.\\\label{5}
&\gl'\gd+\gd\gd'=1-\gl'.\\\label{6}
&\gl\gd'+\gd\gd'=1-\gl.
\end{align}
Adding equation \eqref{1} to equation \eqref{5} gives
$\gl\gl'+\gl'\gd=1-\gd'.$  This and equation \eqref{3}
gives $\gd =\gd'.$  Then, equations \eqref{1} and \eqref{2}
give $\gl=\gl'.$  Finally, equation \eqref{1} gives $\gl^2-\gd^2=\gl-\gd,$
thus either $\gl=\gd,$ or $\gl+\gd=1.$
If $\gl =\gd,$ then equation \eqref{3} gives $\gl\in\{-1, 1/2\},$
and $A$ is as in \cite{HSS}.
When $\gl+\gd=1,$ this is example \ref{flex}(i),
because $ab=\gd a+\gl b.$
\end{proof}

\subsection*{Axes of Jordan type} $ $
\smallskip

The purpose of this section is to classify algebras $A$
generated by two axes of Jordan type.  Hence, throughout this section
$A$ is generated by the axes $a, b$ of Jordan type
$(\gl,\gd)$ and $(\gl',\gd')$ respectively.
Inspired by \cite{HRS} we let
\[
\gs =ab-\gl'a-\gl b.
\]
Recall our notation, from Notation \ref{not00}.

We start with two examples.

\begin{examples}\label{flex}
Let $\gl,\gd\in\ff$ such that $\gl,\gd\notin\{0,1\},$
$\gl+\gd=1,$ and $\gl\ne\gd.$
\medskip

\noindent
(1)\ Let $A$ be  the $2$-dimensional algebra $\ff a+\ff b$
with multiplication defined by
\[
a^2=a,\quad  b^2=b,\quad ab=\gd a+\gl b,\quad ba=\gl a+\gd b.
\]
Then $(ab)a=\gd a+\gl ba=\gd a+\gl^2a+\gl\gd b$
and $a(ba)=\gl a+\gd ab=\gl a+\gd^2a+\gl\gd b.$
Note however that $\gd+\gl^2=\gl+\gd^2.$ So $(ab)a=(ab)a.$
By symmetry $(ba)b=b(ab).$  Next we show that $R_a^2-R_a=L_a^2-L_a$.
We have
\begin{gather*}
a(ab)-ab=a(\gd a+\gl b)-\gd a-\gl b=\gl ab-\gl b=\gl(\gd a+\gl b)-\gl b\\
=\gl\gd a+\gl^2b-\gl b.
\end{gather*}
and
\begin{gather*}
(ba)a-ba=(\gl a+\gd b)a-\gl a-\gd b=\gd ba-\gd b=\gd(\gl a+\gd b)-\gd b\\
=\gd\gl a+\gd^2b-\gd b.
\end{gather*}
Since $\gl^2-\gl=\gd^2-\gd,$ we are done.

Next
\begin{gather*}
((\ga a+\gb b)a)(\ga a+\gb b)=(\ga a+\gb ba)(\ga a+\gb b)=\\
\ga^2 a+\ga\gb ab+\gb\ga (ba)a+\gb^2(ba)b,
\end{gather*}
and
\begin{gather*}
((\ga a+\gb b)(a(\ga a+\gb b))=(\ga a+\gb b)(\ga a+\gb ab)=\\
\ga^2 a+\ga\gb a(ab)+\gb\ga ba+\gb^2b(ab).
\end{gather*}
Since $(ba)a-ba=a(ab)-ab,$ and since $(ab)a=a(ba),$ we see that
$(xa)x=x(ax),$ for all $x\in A.$  By symmetry $(xb)x=x(bx),$ for all $x\in A,$
so $A$ is flexible.

We have
\begin{align*}
&a(a-b)=a-ab=a-(\gd a+\gl b)=\gl(a-b)\text{ and }\\
&b(a-b)=ba-b=\gl a+\gd b-b=\gl(a-b).
\end{align*}
Similarly $(a-b)a=\gd(a-b)=(a-b)b$. (Indeed $(a-b)^2=0.$) We thus
have $A_{\gl,\gd }(a)=A_{\gl,\gd }(b)=\ff(a-b)$, with
$(\ff(a-b))^2=0,$ and of course $A_0(a)=A_0(b)=0.$  Thus the fusion
laws hold for both $a$ and $b$, and both are  axes of Jordan type
$(\gl,\gd)$. Note that the idempotents in $A$ have the form $\ga
a+(1-\ga)b,\ \ga\in\ff.$  Also $\gs=0.$
\medskip

\noindent
(2)\  Let $A =\ff a+\ff b+\ff x,$ with multiplication
defined by $a^2=a,$ \ $b^2=b,\ x^2=0,$ and
\[
ab = ax =xb =\gl x,\qquad ba=xa=bx=\gd x.
\]
 So
\[
\ff x = A_{\gl,\gd },\qquad A_{0,0}=\ff (b-x).
\]

In checking the fusion rules for $a$, we have $(b-x)x = bx \in \ff
x$ and $x^2 = 0,$ and also
$$(b-x)^2 = b^2 -bx -xb + x^2 = b -(\gl + \gd)x =b-x.$$

Likewise for $b$. Thus $a$ and $b$ are axes of
 Jordan type $(\gl,\gd)$ and $(\gd,\gl)$ respectively.

Let us check the eigenvalues of the idempotent $b-x.$ $a(b-x) = 0 =
(b-x)a.$ $(b-x)x = \gd x$, and $x(b-x) = \gl x$. Thus $b-x$ also is
a Jordan axis, of type $(\gd,\gl).$

 To show that $A$ is flexible,
we compute that
\begin{equation*}
\begin{aligned}
&((\ga a+\gb b +\xi x)a)(\ga a+\gb b +\xi x)\\
&=(\ga a+\xi x)(\ga a+\gb b +\xi x)\qquad \eta=\gb\gd+\xi\gd\\
&=\ga^2 a+\ga\gb\gl x+\ga\xi\gl x+ \eta(\ga\gd+\gb\gl)x\\
&=\ga^2 a+\ga\gb\gl x+\ga\xi\gl x+(\ga\gb\gd^2+\gb^2\gl\gd+\ga\xi\gd^2+\xi\gd\gb\gl)x
\end{aligned}
\end{equation*}
and
\begin{equation*}
\begin{aligned}
&(\ga a+\gb b +\xi x)(a(\ga a+\gb b +\xi x))\\
&= (\ga a+\gb b +\xi x)(\ga a+\gc x)\qquad \gc=\gb\gl+\xi\gl\\
&=\ga^2 a+\ga\gb\gd x+\ga\xi\gd x+ \gc(\ga\gl+\gb\gd)x\\
&=\ga^2 a+\ga\gb\gd x+\ga\xi\gd x+(\ga\gb\gl^2+\gb^2\gl\gd+\xi\ga\gl^2+\xi\gl\gb\gd)x
\end{aligned}
\end{equation*}
We must show that
\[
\ga\gb\gl+\ga\xi\gl+\ga\gb\gd^2+\ga\xi\gd^2=\ga\gb\gd+\ga\xi\gd +\ga\gb\gl^2+\xi\ga\gl^2.
\]
But this follows from the fact that $\gl^2-\gl=\gd^2-\gd.$ By
symmetry $(yb)y=y(by),$ for all $y\in A.$ It is easy to check that
$(yx)y=y(xy),$ for all $y\in A,$ and we see that $A$ is flexible.
Also $\gs=\gl x-\gd a-\gl b.$

Let us determine the idempotents in $A$. Suppose $y:=\gamma_1
a+\gamma_2 b+\gamma_3 x \ne 0.$ Then
\[
\begin{aligned}
y^2=\gamma_1^2 a + \gamma_2^2 b+ &
\gamma_3(ax+xa)+\gamma_3(bx+xb)+\gamma_1\gamma_2(ab+ba) \\
&=\gamma_1^2 a + \gamma_2^2 b+ (\gamma_3(\gamma_1 +
\gamma_2)+\gamma_1\gamma_2)x,
\end{aligned}
\]
so $y$ idempotent
implies $\gamma_1,\gamma_2 \in \{0,1\}$ and $\gamma_3 =
\gamma_3(\gamma_1+\gamma_2)+\gamma_1\gamma_2.$ If $\gamma_1 =
\gamma_2 =1$ then $\gamma_3=-1,$ and $a+b-x$ is an idempotent but
$(a+b-x)x = x,$ so $x$ is a 1-eigenvalue; hence $a+b-x$ is not
absolutely primitive. It follows that the only axes are of the form:
\medskip

$a +\gamma x,$ of type $(\gl,\gd)$ since $(a +\gamma x)x = \gl x$
and $x(a +\gamma x) = \gd x$;

$b +\gamma x,$ of type $(\gd,\gl)$.
\medskip

Note that the  subalgebra generated by all axes of type $(\gl,\gd)$
is $\ff a + \ff x,$ and the  subalgebra generated by all axes of
type $(\gd,\gl)$ is $\ff b + \ff x,$ both proper subalgebras of $A$.
\end{examples}

We modify Seress' Lemma   \cite[Lemma~4.3]{HRS}:
\begin{lemma}[General Seress' Lemma]\label{lem seress}
If $a$ is an axis of type $(\gl,\gd)$
then $a(xy) =(ax)y + a(x_0y_0)$ for any $x \in A$ and $y \in \ff a + A_0(L_a).$
In particular $a(xy)\in (ax)y+\ff a.$
Symmetrically, $(yx)a = y(xa) + ({}_{0}{y}\, {}_{0}{x})a,$
for any $x\in A,$ and $y\in \ff a+A_0(R_a).$ In particular, $(yx)a\in y(xa)+\ff a.$
\end{lemma}
\begin{proof}
Since $L_a$ and $R_a$ commute, we have $x_0a\in A_0(L_a),$ and $x_{\gl}a\in A_{\gl}(L_a),$ hence,
\[\begin{aligned}
a(xy)& =  a(\ga_x\ga_y a+\ga_y x_0a+x_0y_0+\ga_y x_{\gl}a+x_{\gl}y_0)
\\ & =\ga_x\ga_y a+a(x_0y_0)+\gl\ga_y x_{\gl}a+\gl x_{\gl}y_0.
\end{aligned}
\]
whereas
\[
(ax)y = (\ga_xa+\gl x_{\gl})(\ga_ya+y_0)=\ga_x\ga_ya+\gl \ga_y x_{\gl}a+\gl x_{\gl}y_0.\qedhere
\]
\end{proof}

Here are some basic properties of $\gs$ and $A.$

\begin{lemma}\label{lem g}
\begin{enumerate}
\item
$\gs =(\ga_b(1-\gl)-\gd')a-\gl b_0=(\ga_a(1-\gd')-\gl)b-\gd' a_0.$

\item
If $b_{\gl,\gd}=0,$ then $ab=ba=0$ and $A=\ff a\oplus \ff b.$

\item
If $b_0=0,$ then $\dim (A)=2.$

\item
$A$ is spanned by $a, b, \gs,$ so $A$ is spanned by $a,b, ab.$

\item
Either $A$ is commutative, and then $\gl=\gd,$ and $\gl'=\gd',$ or
$\gl\ne\gd$ and $\gl'\ne\gd'.$
\end{enumerate}
\end{lemma}
\begin{proof}
(1)
We compute that,
\[
\begin{aligned}
&\gs =(\ga_ba+\gl b_{\gl,\gd})-\gd' a-\gl\ga_ba-\gl b_0-\gl b_{\gl,\gd}=(\ga_b-\gl\ga_b-\gd')a-\gl b_0,\\
&\gs=(\ga_ab+\gd' a_{\gl',\gd'})-\gd' \ga_a b-\gd' a_0-\gd' a_{\gl',\gd'}-\gl b=(\ga_a-\gd'\ga_a-\gl)b-\gd' a_0.
\end{aligned}
\]

(2)  If $b_{\gl,\gd}=0,$ then $ab, ba\in\ff a,$ so $ab=ba=0,$ by \ref{lem ab notin ffa}(i, ii).

(3)  If $b_0=0,$ then $b_{\gl,\gd}\in\ff a+\ff b.$  But then $ab, ba\in\ff a+\ff b,$ so $\dim (A)=2.$

(4)  Let $V =  \ff a + \ff ab +\ff \gs =  \ff a + \ff b + \ff ab.$
First, by \ref{lem basic}(4), $ba\in V.$

Also, by (1), $\gs a= a\gs\in\ff a,$ and $\gs b=b\gs \in\ff b.$
Further, by Seress' Lemma~\ref{lem seress}, $\gs (ab)\in (\gs a)b+\ff a\subseteq\ff a+\ff ab.$
Hence
\[
\gs^2=\gs (ab-\gd'a-\gl b)\in \ff a+\ff b +\ff ab=V.
\]
This shows that $V$ is a subalgebra of $A,$ so $A=V.$

(5)  Suppose that $\gl=\gd.$ Then, by Lemma \ref{lem basic}(5),
$a\in Z(A).$  Thus $a$ commutes with $b$ and $\gs.$  By (1),
also $\gs b=b \gs,$ so by (2) $A$ is commutative.
By Lemma \ref{lem basic}(5), $\gl'=\gd'.$
 \end{proof}

Let us first deal with the non-commutative case of Lemma \ref{lem g}(5).

\begin{prop}\label{prop nc}
Suppose that $\gl\ne\gd, \gl'\ne\gd',$ and $\dim(A)=3.$
Then $A$ is as in Examples \ref{flex}(ii).
\end{prop}
\begin{proof}
Since $\dim(A)=3,$ Lemma \ref{lem g}(2) implies that $A_{\gl,\gd}(a)\ne 0\ne A_{\gl',\gd'}(b).$
Also, $A_0(a)\ne 0\ne A_0(b),$  by Lemma \ref{lem g}(3).

By Lemma~\ref{lem ab notin ffa}(iv), $b_{\gl,\gd}^2= 0,$ so
\[
b b_{\gl,\gd} = (\ga_b a +b_0+ b_{\gl,\gd})b_{\gl,\gd} =
 \ga_b\gl b_{\gl,\gd}+b_0b_{\gl,\gd}\in\ff  b_{\gl,\gd},
\]
so $b b_{\gl,\gd} = \rho b_{\gl,\gd},$ with $\rho \in \{0,1,\gl'\}$. But $\rho \ne 1$ since
otherwise $b_{\gl,\gd} \in \ff b,$ implying $ab = \gl b,$  contrary to Lemma~\ref{lem ab notin ffa}(i).
Suppose $\rho= 0.$ Then
\[
b=b^2=(b-b_{\gl,\gd})^2=(\ga_b a+b_0)^2=\ga_b^2 a+b_0^2.
\]
This is impossible.

Hence $b b_{\gl,\gd} = \gl' b_{\gl,\gd}.$ Thus $A_{\gl',\gd'}(b)=A_{\gl,\gd}(a)=\ff b_{\gl,\gd}.$
Set $y:=b_{\gl,\gd}.$
We have $ab = \ga_b a + \gl y$ and $ab = \ga_a b + k y,$ for some $k\in \ff.$
Since $a,b,y$ are linearily independent,  $\ga_a = \ga_b = 0.$
Hence $ab=\gl y$ and $ba=\gd y.$  We thus have
\[
ab=ay=\gl y,\quad ba=ya=\gd y,\quad by=\gl' y, yb=\gd'y.
\]
Note now that $b-y\in A_0(a),$
so $(b-y)^2=b-(\gl'+\gd')y\in \ff a+A_0(a).$  But then
also
\[
(b-(\gl'+\gd')y)-(b-y)=y-(\gl'+\gd')y\in \ff a+A_0(a).
\]
This forces  $\gl'+\gd'=1.$  By symmetry, $\gl+\gd=1.$

Further,
$y-\frac{\gl'}{\gd}a, y-\frac{\gd'}{\gl}a\in A_0(b),$
so $\frac{\gl'}{\gd}=\frac{\gd'}{\gl},$ or $\gl'\gl=\gd'\gd.$
It follows that
\[
\gl'\gl=(1-\gl')\gd \implies \gd =\gl'(\gl+\gd)=\gl'.
\]
Then $\gd'=\gl,$ and this is Example \ref{flex}(ii).
\end{proof}

\begin{prop}\label{prop gs}
Assume that $\gl=\gd,\ \gl'=\gd'$ and that $\dim(A)=3.$
\begin{enumerate}
\item
$\gs=(\ga_b(1-\gl)-\gl')a-\gl b_0=(\ga_a(1-\gl')-\gl)b-\gl' a_0.$

\item
$a\gs=(\ga_b(1-\gl)+\gl')a$\quad and \quad
$b\gs=(\ga_a(1-\gl')+\gl)b.$

\item
$b_0^2\in A_0(a),$ and $a_0^2\in A_0(b).$

\item
$\ga_b(1-\gl)+\gl'=\ga_a(1-\gl')+\gl:=\gc,$ and $\gs^2=\gc\gs.$

\item
If $\gl=\gl',$ then $\ga_a=\ga_b:=\gvp,$ and $A$ is a
primitive axial algebra of Jordan type $B(\gl,\gvp),$ as given in
\cite[Theorem (4.7), p.~98]{HRS}.
\end{enumerate}
\end{prop}
\begin{proof}
Note that by Lemma \ref{lem g}(5), $A$ is commutative.

(1)  Follows from \ref{lem g}(1).

(2) Is immediate from (1).

(3)\&(4)  By Seress' lemma \ref{lem seress}, and since
 $\gs_{0}=-\gl b_0,$ we have
\[
(ab)\gs=a(b\gs)-\gl a(b_0^2)=(\ga_a(1-\gl')+\gl)ab-\gl a(b_0^2).
\]
Similarly,
\[
(ab)\gs=(ba)\gs=(\ga_b(1-\gl)+\gl')ab-\gl' b(a_0^2).
\]
Since $a(b_0^2)\in \ff a$ and  $b(a_0^2)\in \ff b$ and $a,b$ and
$ab$ are linearly independent, we must have $a(b_0^2)=b(a_0^2)=0,$
so (3) follows and the first part of (4) follows. Now we have $(ab)\gs =\gc ab,$ so
\[
\gs^2=(ab-\gl'a-\gl b)\gs =(ab)\gs -\gl' a\gs-\gl b\gs=\gc ab-k\gl' a-k\gl b=\gc\gs.
\]

(5)  If $\gl=\gl',$ then by (4), $\ga_a=\ga_b:=\gvp,$ and $A$ is the algebra $B(\gl,\gvp).$
\end{proof}

\begin{cor}\label{co 020}
${A_0(a)}^2 \subseteq A_0(a)$, that is, condition (c) in the
introduction of \cite{HRS} follows from conditions (a),(b), and (d).
\end{cor}
\begin{proof}
By Lemma \ref{lem g}, $\dim (A)\le 3.$
If $\dim (A)=2,$ the we are done by Theorem \ref{thm dim 2}.
So assume $\dim (A)=3.$  Of course $\dim (A_0) \le 2.$

If $A$ is not commutative, then Proposition \ref{prop nc} completes the proof.
So we may assume that $A$ is commutative.
Suppose $b_0=0.$ then $b_{\gl,\gd}\in \ff a+\ff b,$ and then $ab, ba\in\ff a+\ff b,$
and $\dim (A)=2,$ a contradiction.
Suppose $\dim (A_0) = 1$ then
$A_0= \ff b_0,$ and we are done  by Proposition \ref{prop gs}(3).

Hence we may assume that $\dim (A_0) = 2,$ so $A =  \ff a+A_0.$ Hence
$b =\ga_b a+b_0,$  so $ab=\ga_ba,$ contradicting Lemma \ref{lem ab notin ffa}(i).
\end{proof}

In the next result we will use the {\it Miyamoto involution}
$\gt_a,$ associated with an axis $a$ of Jordan  type $(\gl,\gd).$
Because of the $Z_2$-grading of $A$ induced by $a,$ the map
$x\mapsto x^{\gt_a},$ where $x=\ga_xa+x_0+x_{\gl,\gd},$ and
$x^{\gt_a}=\ga_xa+x_0-x_{\gl,\gd}$ is an automorphism of $A.$

\begin{prop}\label{prop c}
Assume that $\gl=\gd,\ \gl'=\gd',$ and that $\dim (A)=3.$
Write $A_0$ for $A_0(a).$
\begin{enumerate} \eroman
\item
we have
\[
\gc:=\ga_b(1-\gl)+\gl'=(\ga_a(1-\gl')+\gl)\in \ff,
\]
and $A=\ff a + \ff b + \ff \gs,$ with the
multiplication table $a^2 = a,$ $b^2= b,$ $ab=ba= \gl' a +\gl b+\gs,$
$a\gs =\gs a =\gc a,$ $b \gs = \gs b=\gc b$,  and $\gs^2 = \gc \gs$.

Conversely, if $A$ is an algebra defined by this multiplication
table, the $A$ is generated by $a,b,$ which are axes of Jordan
type $(\gl,\gl), (\gl',\gl'),$ respectively.

\item
If $\gc=0,$ then $\gl = \gl'$.

\item
If  $ \gamma \ne 0,$ then  $ \frac{ 1}{\gc}\gs, $ denoted as ${\bf 1},$ is
the unit element of $A$.  We have $\dim (A_0) = 1=\dim (A_0(b)).$

\item
${\bf 1}-a$ is an axis of Jordan type  $(1-\gl,1-\gl).$

\item
For axes $c,d\in A,$ let $A'(c,d)$ be the subalgebra generated by $c$ and $c^{\gt _d}.$
If $\dim (A'(a,b)) = 2,$ then either $\gl = \half,$ or $A=A'({\bf 1}-a,b).$

\item
$A$ is a primitive axial algebra of Jordan type as in \cite{HRS}.
\end{enumerate}
\end{prop}
 \begin{proof}
(i)  The first part is Proposition \ref{prop gs}(2,3,5).
Conversely, we have
\[
(L_a-1)L_a(\gs) =  (L_a -1)(\gc a) = 0,
\]
of course $L_a-1L_a(a)=0,$ and
\[
(L_a-1)L_a(L_a-\gl)(b)=(L_a-1)L_a(ab-\gl b)=(L_a-1)L_a(\gs-\gl' a)=0,
\]
So $a$ is an axis, and, by symmetry, so is $b$.
Clearly $a,b$ generate $A,$ so $A$ is as claimed.
\medskip

(ii)
Suppose $\gc =0,$ then $\gs x=0,$ for all $x\in A.$ We have
\[
ab-\gs=\gl' a+\gl b \implies a(ab)=\gl' a+\gl ab.
\]
by Lemma \ref{lem basic}(3), $ a(ab)=\ga_b (1-\gl) a+\gl ab.$  It follows that
$\gl'= \ga_b (1-\gl).$  By symmetry, $\gl= \ga_a (1-\gl').$  But now
Proposition \ref{prop gs}(4) shows that $2\gl'=2\gl.$
\medskip

(iii)
If $\dim (A_0) =2$ then $A = A_0 +\ff a,$ implying $ab \in \ff
a,$  contrary to~Lemma~\ref{lem ab notin ffa}.   By symmetry $\dim (A_0(b))=1.$
\medskip

(iv) This is easy to check.
\medskip

(v)
Suppose that $\dim (A'(a,b))=2,$ and $\gl\ne\half.$  By
\cite[Lemma 3.1.2, p.~269]{HSS} we either have $aa^{\gt_b}=0,$
or $\gl=-1.$  In the first case we have $({\bf 1}-a)a^{\gt_b}=a^{\gt_b},$
contradicting the primitivity of ${\bf 1}-a.$  In the second case,
if $\charc(\ff)=3,$ then $-1=\half,$ otherwise
we get that ${\bf 1}-a$ is an axis of type $(2,2),$
and then necessarily $\dim A'({\bf 1}-a, b)=3.$
\medskip

(vi)
By Proposition \ref{prop gs}(5) it suffices
to show that $A$ is generated by two axes of the same Jordan type.
Note that $A'(c,d)$ is such an algebra for axes $c, d\in A.$
If $A'(a,b)=A,$ we are done.  Otherwise, by (v), $\gl=\half.$
Similarly, considering $A'(b,a),$ we are done unless $\gl'=\half.$
But now $\gl=\gl'=\half.$
\medskip
\end{proof}

We can finally prove

\begin{thm}\label{thm 2-gen}
Either $A$ is an
axial algebra of Jordan type, as in \cite{HSS},
or $A$ is as in one of the two algebras in Examples \ref{flex}.
In particular, $A$ is flexible.
\end{thm}
\begin{proof}
If $\dim (A)=2,$ then this is Proposition \ref{thm dim 2}.
Hence, by Lemma \ref{lem g}, we may assume that $\dim(A)=3.$
Now Lemma \ref{lem g}(5) together with Propositions \ref{prop nc}
and \ref{prop c} complete the proof.

\end{proof}

\end{document}